\documentclass[a4paper, 12pt,oneside,reqno]{amsart}
\usepackage[a4paper]{geometry}
\usepackage{amssymb,amsmath,amsthm}
\usepackage[T1]{fontenc}
\usepackage{lmodern,xcolor}
\usepackage[breaklinks=true]{hyperref}
\hypersetup{hidelinks}

\usepackage{graphicx}

\DeclareSymbolFont{symbolsC}{U}{pxsyc}{m}{n}
\SetSymbolFont{symbolsC}{bold}{U}{pxsyc}{bx}{n}
\DeclareFontSubstitution{U}{pxsyc}{m}{n}
\DeclareMathSymbol{\medcirc}{\mathbin}{symbolsC}{7}
\DeclareMathSymbol{\medbullet}{\mathbin}{symbolsC}{8}

\def\Var{\mathrm{Var}}

\def\Der{\mathrm{Der}}
\def\Nov{\mathrm{Nov}}
\def\Acaa{\mathrm{Acaa}}
\def\ComMag{\mathrm{ComMag}}
\def\Com{\mathrm{Com}}
\def\As{\mathrm{As}}
\def\Lie{\mathrm{Lie}}
\def\wt{\mathrm{wt}}
\def\pt{\mathrm{pt}}
\def\Res {\mathop {\fam 0 Res}}

\def\End {\mathop {\fam 0 End}\nolimits}
\def\Hom {\mathop {\fam 0 Hom}\nolimits}

\def\vecotimes{\mathbin{\vec\otimes}}

\let\<\langle
\let\>\rangle

\newtheorem{definition}{Definition}
\newtheorem{lemma}{Lemma}
\newtheorem{proposition}{Proposition}
\newtheorem{theorem}{Theorem}
\newtheorem{corollary}{Corollary}
\newtheorem{remark}{Remark}
\newtheorem{example}{Example}

\title[The Weight Criterion, Dong Property, and Manin Products of Operads]{The Weight Criterion, Dong Property, and Manin Products of Operads}

\author{P. Kolesnikov}

\address{Sobolev Institute of Mathematics, Novosibirsk, Russia}

\email{pavelsk77@gmail.com}

\author{B. Sartayev$^{*}$}

\address{Narxoz University, Almaty, Kazakhstan}

\address{SDU University}

\email{baurjai@gmail.com}



\thanks{${}^{*}$Corresponding author: Bauyrzhan Sartayev   (baurjai@gmail.com)}

\begin{document}

\begin{abstract}
We establish that, for a binary operad $\Var$, the Hadamard product of $\Var$ and the operad $\Nov$ of Novikov algebras is a binary operad (i.e., coincides with their Manin white product $ \Nov \medcirc \Var $) if and only if the quadratic part of $\Var$ governs a 2-variety. The condition $\Nov\medcirc \Var =  \Nov \otimes \Var $ is known to be equivalent to the existence of an easily verifiable characterization of those differential $\Var$-polynomials which lie in the free derived $\Var$-algebra. For a binary quadratic operad, these conditions are also equivalent to the Dong property of formal distributions. We distinguish the white product defined as a generated sub-operad from its quadratic approximation and give an example in which the canonical epimorphism between them is not injective.
\end{abstract}

\maketitle

\section{Introduction}

A Novikov algebra is a linear space equipped with a bilinear operation $\circ$ satisfying the following identities:
\begin{gather}
(x\circ y)\circ z - x\circ (y\circ z) - (y\circ x)\circ z + y\circ (x\circ z) =0, \label{eq:LSym} \\
(x\circ y)\circ z - (x\circ z)\circ y = 0. \label{eq:RCom}
\end{gather}
Such algebras first appeared in \cite{GD1979} in the following context: the coefficients of a certain linear operator $H$ emerging in formal variational calculus are structure constants of a Novikov algebra if and only if the operator $H$ is Hamiltonian.
Later in \cite{BalNov1985} the same variety of algebras emerged in a different problem related to the theory of partial differential systems of hydrodynamic type. Namely, the coefficients of a formal bracket on the space of field functions form the structure tensor of a Novikov algebra if and only if the bracket satisfies the Jacobi identity (i.e., this is a generalized Poisson bracket, the main component of a Hamiltonian formalism in partial differential equations).
In \cite{Osborn92}, algebras satisfying these identities were called Novikov algebras.

The structure and representation theory of Novikov algebras have been widely studied (see \cite{ZP2024} and references therein). The combinatorial study of Novikov algebras was initiated in \cite{DzhLofwall2002}, where it was shown, in particular, that the free Novikov algebra generated by a set $X$ is a subalgebra of the free commutative associative algebra $\Com\Der\<X,d\>$ generated by $X$ with one derivation $d$, relative to the operation
\[
f\circ g = f d(g),\quad f,g\in \Com\Der\<X,d\>.
\]

It was already noted in \cite{GD1979} that every associative and commutative algebra $A$ ($A\in \Com $) equipped with a derivation $d$ turns into a Novikov algebra (denoted $A^{(d)}$) with the new operation  $a\circ b = ad(b)$, $a,b\in A$. Later in \cite{BCZ2017} it was proved that every Novikov algebra can be embedded into an algebra of the form $A^{(d)}$ for appropriate $A\in \Com $ and~$d\in \Der(A)$.

The result of \cite{DzhLofwall2002} thus may be interpreted as follows. Consider the T-ideal $I_{D\Com }$ of all identities that hold on all algebras of the form $A^{(d)}$, where $A$ is an associative and commutative algebra, $d$ is a derivation of $A$, $d\in \Der(A)$. Then $I_{D\Com}$ is generated as a T-ideal by the identities \eqref{eq:LSym} and  \eqref{eq:RCom}.

In \cite{KolSartOrazg2019} the following question was considered: what are the analogues of \eqref{eq:LSym} and \eqref{eq:RCom} if we replace  $\Com$ with another variety $\Var $ of algebras with binary operations?

Namely, suppose $A$ is a (non-associative) algebra with one or more bilinear operations $(x,y)\mapsto xy$ and $d$ is a derivation of~$A$. Then the {\em derived} algebra $A^{(d)}$  is the same space equipped with new operations
\[
x\prec y =x d(y),\quad x\succ y = d(x)y.
\]
Given a variety $\Var $ of algebras, the derived variety $D\Var$ is generated (as an HSP-class) by all derived algebras $A^{(d)}$ for all $\Var$-algebras $A$ with all possible derivations $d\in \Der(A)$. The problem is to find the identities generating the T-ideal $I_{D\Var}$.

If $\Var $ is defined by multilinear identities, then it is natural to state the question in terms of operads.

It was shown in \cite{KolSartOrazg2019} that the class $D\Var$ is governed by the operad $\Nov\medcirc \Var$, where $\medcirc$ is the Manin white product of operads in the sense of Ginzburg and Kapranov \cite[2.2.3]{GK1994}. To be more precise, $\Nov\medcirc \Var$ is the sub-operad of the Hadamard product $\Nov\otimes \Var$ generated by $\Nov(2)\otimes\Var(2)$.

The main purpose of this paper is to determine when the weight criterion holds for an arbitrary variety with finitely many binary operations. The criterion says that a differential $\Var$-polynomial $f\in \Var\Der\langle X,d\rangle \cong \Var\langle X\cup X'\cup\dots \rangle $ belongs to the free derived algebra
$D\Var\langle X\rangle $
if and only if it is weight-homogeneous of weight $-1$, where
the weight is additive on products and satisfies $\wt(x^{(s)})=s-1$ for $x\in X$; see \cite{KolSartSIGMA2024}.
It is known that the weight criterion is applicable
if and only if the Hadamard product $\Var\otimes\Nov$
coincides with its sub-operad generated by $\Var(2)\otimes\Nov(2)$; see Proposition~\ref{prop:WeightHadamard}. Several particular cases of this equality were considered in \cite{KolSartSIGMA2024,KolSartOrazg2019,SK_Europ}.

We show that the equality $\Nov\medcirc \Var = \Nov \otimes \Var$ is determined by the identities of degree~3 of $\Var$. More precisely, every product $(x_1\circ_j x_2)\circ_i x_3$ must be expressible as a linear combination of products in which $x_3$ occurs in the inner multiplication (Theorem~\ref{thm:WeightCriterion}). This condition is equivalent to the following ideal-theoretic property: for every algebra $A$ of the variety and every ideal $J$ of $A$, the space $J^2$ is again an ideal of $A$. A variety with this property is called a 2-variety. Thus the weight criterion holds precisely when the quadratic part of $\Var$ governs a 2-variety (Proposition~\ref{prop:TwoVariety}). In particular, relations of higher degree do not affect the answer.

For a binary quadratic operad $\Var $, the same condition has a dual
interpretation. It means that the compositions of two binary
generators with a fixed pair of inner arguments are linearly
independent in the Koszul dual operad. By the criterion of
\cite{KolSart_DongOperads}, this is exactly the case when a natural analogue of the Dong Lemma from the theory of vertex and conformal algebras \cite{Kac1} holds for distributions over $\Var$-algebras (i.e., $\Var $ satisfies the {\em Dong property}).
We obtain the equivalence between the Dong property and the weight criterion in Theorem~\ref{thm:DongWeight}.

The terminology of Manin white products requires some care. The binary sub-operad of the Hadamard product need not be quadratic even when both factors are quadratic. We distinguish this construction from the operad defined by its quadratic relations in Section~\ref{sec:Preliminaries}. An explicit example in which these operads differ already in arity~4 is given below. Throughout the paper, the product governing a derived variety is the sub-operad of $\Var\otimes\Nov$ generated by $\Var(2)\otimes\Nov(2)$.

The Dong property need not pass from $\Var$ to $D\Var$. In Section~\ref{sec:Dong} we characterize the binary operads $\Var$ for which $D\Var$ is a Dong operad (Theorem~\ref{thm:DerivedDong}): this happens if and only if the compositions of two binary generators with a fixed outer argument span $\Var(3)$, equivalently, if and only if $\Var\medcirc\mathcal Q=\Var\otimes\mathcal Q$ for every binary operad $\mathcal Q$. In particular, $D^2\Var$ is a Dong operad only when $\Var(3)=0$ (Corollary~\ref{cor:DerivedIterations}); for $\Var=\Com$ the chain $\Com$, $\Nov=D\Com$, $D\Nov$ shows that the property may survive exactly one derivation. Finally, we show that the derived operad $D\Acaa$ is not quadratic, although the weight criterion holds for $\Acaa$ (Proposition~\ref{thm:DerivedAcaaNonquadratic}).

\section{Preliminaries in operads and Manin products}\label{sec:Preliminaries}

Throughout the paper, all spaces are over a field $\Bbbk$ of
characteristic~0. All algebras are non-unital, and all spaces of
binary generators are assumed to be finite-dimensional.

Recall that an {\em operad}  $\mathcal P$
is a family of linear spaces $\mathcal P(n)$, $n\ge 1$, where each $\mathcal P(n)$ is a module over the symmetric group $S_n$, equipped with a {\em composition rule}
\[
\mathcal P(n)\otimes \mathcal P(m_1)\otimes \dots\otimes \mathcal P(m_n)
\to \mathcal P(m_1+\dots +m_n),
\quad n,m_i\ge 1,
\]
satisfying natural axioms including associativity, unitality, and equivariance (see {\em symmetric operad} in \cite{LodayVallette2012}).
A {\em morphism} of operads $f: \mathcal P\to \mathcal Q$ is a family of $S_n$-invariant linear maps $f(n):\mathcal P(n) \to \mathcal Q(n)$,
$n\ge 1$,
preserving the composition rule.

For example, given a linear space $A$, the operad $\End_A$ consists of linear spaces
$\End_A(n) = \Hom (A^{\otimes n}, A)$, $n\ge 1$,
equipped with the ordinary composition rule. The symmetric group $S_n$ acts on $\End_A(n)$ by permutation of arguments.

Suppose $\Var $ is a variety of algebras with bilinear operations defined by multilinear identities. Then there exists an operad also denoted $\Var$ such that a structure of a $\Var$-algebra on a linear space $A$ is the same as a morphism of operads $\Var \to \End_A$. Namely, the space
$\Var(n)$ consists of all multilinear polynomials in the free $\Var$-algebra in the independent variables $x_1,\dots, x_n$, the symmetric group $S_n$ acts by permutations of variables, and the composition is naturally defined via the substitution (see, e.g., \cite{LodayVallette2012}).

An operad $\mathcal P$ is said to be {\em binary} if
$\dim \mathcal P(1)=1$ and the entire $\mathcal P$ is generated (as an operad) by the space $\mathcal P(2)$.

Given an $S_2$-space $V$, $\dim V<\infty$, one may define the free (binary) operad $\mathcal F_V$ generated by $V$. In particular,
$\mathcal F_V(2)=V$, $\mathcal F_V(3) = \mathrm{Ind}_{\Bbbk S_2}^{\Bbbk S_3} (V\vecotimes V)$,
where $V\vecotimes V$ is the tensor square of $V$ such that $S_2$ acts on the second factor only.
Suppose $R$ is an $S_3$-subspace of $\mathcal F_V(3)$. Then the quotient operad of $\mathcal F_V$ modulo the operad ideal generated by $R$ is denoted $\mathcal P(V,R)$; such operads are called {\em quadratic} since their defining relations are of degree two with respect to the generators~$V$.

With representatives $1,(13),(23)$ of the
cosets of $S_2$ in $S_3$, we have a direct sum of vector spaces
\begin{equation}\label{eq:ThreeSummands}
\mathcal F_V(3)=U\oplus(13)U\oplus(23)U,
\qquad U=V\vecotimes V.
\end{equation}

For an arbitrary binary operad $\mathcal P$, consider $V=\mathcal P(2)$ and put
\[
R_{\mathcal P}=
\ker\bigl(\mathcal F_{V}(3)\to\mathcal P(3)\bigr).
\]
The \emph{quadratic part} of $\mathcal P$ is
the operad
$\mathcal P_q=\mathcal P(V,R_{\mathcal P})$.
There is a canonical epimorphism
$\mathcal P_q\to\mathcal P$,
which is an isomorphism in arities $n=1,2,3$.

We also recall the convention for Koszul duality. Set
$V^\vee=V^*\otimes\operatorname{sgn}_2$.
The spaces $\mathcal F_{V^\vee}(3)$ and $\mathcal F_V(3)$ have
the standard nondegenerate sign-equivariant pairing: different
summands in \eqref{eq:ThreeSummands} are orthogonal, and for dual
bases $(\mu_i^\vee)$ and $(\mu_i)$ it is given by
\[
\bigl\langle
\sigma(\mu_i^\vee\vecotimes\mu_j^\vee),
\tau(\mu_k\vecotimes\mu_l)
\bigr\rangle
=\delta_{\sigma,\tau}\operatorname{sgn}(\sigma)
\delta_{ik}\delta_{jl},
\qquad \sigma,\tau\in\{1,(13),(23)\}.
\]
For $\mathcal P=\mathcal P(V,R)$, its Koszul dual operad is
$\mathcal P^!=\mathcal P(V^\vee,R^\perp)$; see
\cite{GK1994,LodayVallette2012}.

We identify the elements of $\mathcal F_V(n)$ with multilinear polynomials in formal variables $x_1,\dots, x_n$ with binary operations $v\in V$.
In this way, we identify $u\vecotimes v$ with the composition $u(v(x_1,x_2),x_3)$, $(12)(u\vecotimes v) = u\vecotimes (12)v$ with $u(v(x_2,x_1),x_3)$, etc.

\begin{example}\label{exmp:Acaa}
Let $V$ be the skew-symmetric 1-dimensional $S_2$-module spanned by $\mu =-(12)\mu$.
If $R \subset \mathcal F_V(3)$ is generated as an $S_3$-module by
\[
\mu\vecotimes \mu + (13)(\mu\vecotimes \mu)
\]
then $\mathcal P(V,R)$ is the operad of $\Acaa$-algebras (anti-commutative and anti-associative): the relation reads $(x_1x_2)x_3+x_1(x_2x_3)=0$.
\end{example}

\begin{example}\label{exmp:Nov}
Let $V$ be a 2-dimensional $S_2$-module spanned by $\nu_1$ and $\nu_2=(12)\nu_1$. Consider
$R\subset \mathcal F_V(3)$ generated by
\[
(\nu_1\vecotimes \nu_1) -(\nu_1\vecotimes \nu_2)
-(13)(\nu_2\vecotimes \nu_2)+(23)(\nu_2\vecotimes \nu_1),
\]
\[
(\nu_1\vecotimes \nu_1) - (23)(\nu_1\vecotimes \nu_1).
\]
These relations represent left symmetry and right commutativity of the operation $\nu_1=x_1x_2$,
so $\mathcal P(V,R)$ is exactly the operad $\Nov $
governing the variety of Novikov algebras.
\end{example}

Given two operads $\mathcal P$ and $\mathcal Q$,
their {\em Hadamard product} $\mathcal P\otimes \mathcal Q$ is defined by the family of spaces
$\mathcal P(n)\otimes \mathcal Q(n)$, $n\ge 1$,
with component-wise composition and symmetric action.
However, even if both $\mathcal P$ and $\mathcal Q$
are binary operads, $\mathcal P\otimes \mathcal Q$ need not be binary. The notion of the {\em Manin white product} $\mathcal P\medcirc \mathcal Q$ resolves this issue: this is the sub-operad of
$\mathcal P\otimes \mathcal Q$ generated by
the space $\mathcal P(2)\otimes \mathcal Q(2)$ \cite[2.2.3]{GK1994}.

The Manin white product defined in this way need not be quadratic: Proposition~2.2.4 of \cite{GK1994} does not hold in this generality; see Proposition~\ref{thm:DerivedAcaaNonquadratic} below.
Therefore, for binary quadratic operads $\mathcal P$ and $\mathcal Q$, we distinguish the following two constructions:
\begin{enumerate}
    \item $\mathcal P\medcirc_{GK} \mathcal Q$, the binary sub-operad of $\mathcal P\otimes \mathcal Q$;

    \item $\mathcal P\medcirc_R \mathcal Q$, the quadratic part of
    $\mathcal P\medcirc_{GK} \mathcal Q$,
    i.e., the operad generated by
    $\mathcal P(2)\otimes \mathcal Q(2)$
    relative to only quadratic relations that hold on $\mathcal P\otimes \mathcal Q$.
\end{enumerate}
To describe the second construction precisely, let
$E=\mathcal P(2)$, $F=\mathcal Q(2)$, and let
\[
\Delta :\mathcal F_{E\otimes F}
\longrightarrow\mathcal F_E \otimes\mathcal F_F
\]
be the morphism of operads which sends the generators from $E\otimes F$ identically to themselves, i.e., $\Delta(2)=\mathrm{id}_{E\otimes F}$.
The composition of $\Delta $ with the epimorphisms
$\mathcal F_E\to \mathcal P$ and $\mathcal F_F\to \mathcal Q$ is a morphism $\Theta : \mathcal F_{E\otimes F}\to \mathcal P\otimes \mathcal Q$.
Put
\begin{equation}\label{eq:WhiteRelations}
R_\circ
=\ker\Theta(3)
=\Delta(3)^{-1}\bigl(
R_{\mathcal P}\otimes\mathcal F_F(3)
+\mathcal F_E(3)\otimes R_{\mathcal Q}\bigr).
\end{equation}
Then $\mathcal P\medcirc_R\mathcal Q=\mathcal P(E\otimes F,R_\circ)$.
This is the quadratic Manin construction; see
\cite{Vallette2008}.

\begin{example}\label{exmp:White_R_Nov_Acaa}
In order to compute $\Nov\medcirc_R \Acaa$, one may consider the two-dimensional space of generating operations with a basis $\mu_1 = \nu_1 \otimes \mu$,
$\mu_2 = -\nu_2\otimes \mu = (12)\mu_1$, and find all linear combinations of
$(\mu_i\vecotimes \mu_j)$,
$(13)(\mu_i\vecotimes \mu_j)$,
$(23)(\mu_i\vecotimes \mu_j)$,
$i,j=1,2$,
that vanish in $\Acaa(3)\otimes \Nov(3)$.
It is straightforward to see that $R_\circ$ is generated as an $S_3$-module by
\[
\begin{gathered}
\mu_1\vecotimes\mu_1 + (23)(\mu_1\vecotimes \mu_1), \\
\mu_1\vecotimes\mu_1 + (13)(\mu_2\vecotimes \mu_2)
+ \mu_1\vecotimes\mu_2 + (23)(\mu_2\vecotimes \mu_1).
\end{gathered}
\]
\end{example}

Our notation $\medcirc_{GK}$ refers specifically
to the generated sub-operad in \cite[2.2.3]{GK1994}.
Proposition~2.2.4 of \cite{GK1994} would identify the two constructions $\medcirc_R$ and $\medcirc_{GK}$, but the canonical epimorphism
\begin{equation}\label{eq:WhiteEpi}
\mathcal P\medcirc_{R} \mathcal Q
\to \mathcal P\medcirc_{GK} \mathcal Q
\end{equation}
is not necessarily injective (see Proposition~\ref{thm:DerivedAcaaNonquadratic} below).
Both operads $\mathcal P\medcirc_{R} \mathcal Q$ and $\mathcal P\medcirc_{GK} \mathcal Q$ have generators $E\otimes F$ and the same defining relations in arity~3,
but the second one may have independent relations in arity $n>3$.

We will be interested in the Manin white product of $\Nov$ and $\Var$ for various operads $\Var$ corresponding to varieties of algebras. In many cases, the equality $\Nov\medcirc_{R}\Var=\Nov\medcirc_{GK}\Var$ can be established by other methods. The quadratic white product is considered only when $\Var$ is quadratic. From now on, an unsubscripted $\medcirc$ means $\medcirc_{GK}$.

\section{Derived algebras and derived operads}\label{sec:Derived}

Let $A$ be an algebra, i.e., a linear space over a field $\Bbbk $ equipped with a family of bilinear products $\circ_i$, $i\in I$.

A variety of algebras is the class of all algebras (with a fixed language set $I$) satisfying a family of polynomial identities. The set of all such identities is a T-ideal in the corresponding free algebra. We will consider only multilinear varieties, for which the T-ideal is generated  by multilinear identities (this is not a restriction if $\mathrm{char}\,\Bbbk =0$).

Recall that a linear map $d: A\to A$ is a derivation of the algebra if the Leibniz rule holds for all products, i.e.,
\[
d(a\circ_i b) = d(a)\circ_i b + a\circ_i d(b),
\]
for all $a,b\in A$, $i\in I$.
The set of all derivations of an algebra $A$ is denoted $\Der(A)$.

Given a pair $(A,d)$, where $A$ is an algebra and $d\in \Der(A)$, we define the {\em derived algebra} $A^{(d)}$ on the same linear space $A$, with two operations for each $i\in I$:
\[
a\prec_i b = a\circ_i d(b),\quad  a\succ_i b = d(a)\circ_i b,\quad i\in I,
\]
for $a,b\in A$.

Suppose $\Var $ is a variety of algebras in the language $I$. Consider the variety $\Var\Der$ of all $\Var$-algebras equipped with a derivation, regarded as an additional unary operation. The elements of $\Var\Der$ are pairs $(A,d)$, where $A\in\Var$ and $d\in \Der(A)$. The {\em derived variety} $D\Var $ is the class of algebras satisfying all the identities that hold on every $A^{(d)}$ for all $(A,d)\in \Var\Der$. In other words, $D\Var $ is the variety generated by all derived algebras $A^{(d)}$, $(A,d)\in \Var\Der$.

\begin{example}[\cite{DzhLofwall2002}]\label{exmp:DCom=Nov}
For $\Var =\Com$, the variety of commutative associative algebras, the derived variety $D\Com $ coincides with the class of Novikov algebras.
\end{example}

\begin{example}[\cite{KolSartOrazg2019}]
For $\Var=\Lie$, the variety of Lie algebras, the derived variety $D\Lie $ coincides with the class of all nonassociative algebras with one operation.
\end{example}

In both examples above it is enough to consider only one operation since if the initial product is (anti-)commutative, then $a\prec b =\pm b\succ a$ in the derived algebra.

\begin{remark}\label{rem:Speciality}
Although the free algebra $D\Var\<X\>$ is {\em special}, i.e., isomorphic to a subalgebra of a derived algebra from $\Var\Der $, the same statement need not hold for an arbitrary algebra from $D\Var $. For example, there exist non-special algebras from $D\text{pre-}\Com$ (derived variety of Zinbiel algebras) \cite{KolSartSIGMA2024}.

On the other hand, all algebras from $D\Com$, $D\As$, $D\Lie$ are special in the above-mentioned sense (see, e.g., \cite{SK_Europ}). For the commutative and Lie cases, see also \cite{BCZ2017,KolSartOrazg2019}, respectively.
\end{remark}

Recall that the free algebra in the variety $\Var\Der$ generated by a set $X$ is isomorphic (as a $\Var$-algebra) to the free algebra $\Var\<X^{(\omega)}\>$, where $X^{(\omega)}=X\cup X'\cup\dots\cup X^{(n)}\cup\dots$. The derivation $d$ on $\Var\<X^{(\omega)}\>$ is determined by the rule $d:x^{(n)}\mapsto x^{(n+1)}$, $x\in X$, $n\ge0$. The free algebra $D\Var\<X\>$ is naturally identified with the subalgebra of the derived algebra $\Var\<X^{(\omega)}\>^{(d)}$ generated by $X\subset X^{(\omega)}$. This follows from the universal property of the free differential $\Var$-algebra.

In order to recognize which elements from $\Com\<X^{(\omega )}\>$ belong to $D\Com\<X\>=\Nov\<X\>$, the following criterion was proposed in \cite{DzhLofwall2002}. For a generator $x^{(n)}$ in the free commutative algebra $\Com\<X^{(\omega )}\>$, define its {\em weight} as
\[
\wt(x^{(n)}) = n-1,\quad x\in X,\ n\in \mathbb Z_+.
\]
Then extend the weight function to monomials by the additive rule, i.e.,
\[
\wt(a_1\dots a_k) = \wt(a_1)+\dots + \wt(a_k),
\quad a_i\in X^{(\omega )}.
\]
A polynomial $f\in \Com\<X^{(\omega )}\>$
is said to be weight-homogeneous if all monomials
that appear in $f$ have equal weights.

\begin{proposition}[\cite{DzhLofwall2002}]\label{prop:WeightNov}
A polynomial $f\in \Com\<X^{(\omega )}\>$
belongs to $\Nov\<X\>$ if and only if
$f$ is weight-homogeneous of weight $-1$.
\end{proposition}

The same definition of weight applies in $\Var\<X^{(\omega )}\>$ for an arbitrary multilinear variety $\Var$.
If $f\in D\Var\<X\>$, then $f$ is a weight-homogeneous element of $\Var\<X^{(\omega )}\>$ and $\wt(f)=-1$.
As shown in \cite{KolSartSIGMA2024}, if the converse is true
(i.e., if the {\em weight criterion} works for $\Var $)
then every $D\Var $-algebra is special.
(This is just a sufficient condition for every
$D\Var $-algebra to be embeddable into a differential
$\Var $-algebra.)

For $n\ge1$, denote by $W_\Var(n)$ the space of differential
polynomials which are multilinear in $x_1,\dots,x_n$ and have weight
$-1$. Here multilinearity means that, for each $i$, exactly one of the
letters $x_i^{(s)}$, $s\ge0$, occurs in each monomial. As noted above,
$D\Var(n)\subseteq W_{\Var}(n)$.

\begin{proposition}[cf. {\cite[Lemma 3.1, Remark 3.3]{KolSartSIGMA2024}}]\label{prop:WeightHadamard}
For every binary operad $\Var$ there are natural $S_n$-module isomorphisms
\[
W_\Var(n)\cong\Nov(n)\otimes\Var(n),\qquad n\ge1,
\]
under which $D\Var(n)$ corresponds to
$(\Nov\medcirc\Var)(n)$. Consequently, the following conditions
are equivalent:
\begin{itemize}
\item[(i)] $\Nov\medcirc\Var=\Nov\otimes\Var$;
\item[(ii)] for every set $X$, the algebra $D\Var\<X\>$ consists
exactly of the weight-homogeneous polynomials of weight $-1$ in
$\Var\<X^{(\omega)}\>$.
\end{itemize}
\end{proposition}

\begin{proof}
By Proposition~\ref{prop:WeightNov}, the differential commutative
monomials
\[
x_1^{(s_1)}\cdots x_n^{(s_n)},\qquad
s_i\ge0,\quad s_1+\cdots+s_n=n-1,
\]
form a basis of $\Nov(n)$. For each fixed tuple $(s_1,\dots,s_n)$,
the multilinear component of $\Var\<X^{(\omega)}\>$ on the letters
$x_1^{(s_1)},\dots,x_n^{(s_n)}$ is a copy of $\Var(n)$.
The components belonging to different tuples form a direct sum,
since the ideal of identities of $\Var$ is multihomogeneous in
the alphabet $X^{(\omega)}$. Therefore the map
\[
x_1^{(s_1)}\cdots x_n^{(s_n)}\otimes f(x_1,\dots,x_n)
\longmapsto f(x_1^{(s_1)},\dots,x_n^{(s_n)})
\]
is an isomorphism onto $W_\Var(n)$. It is equivariant for the
simultaneous permutation of variables.

Under these maps, a tensor of binary operations corresponds to
one of the operations $\prec_i$, $\succ_i$.
Composition corresponds to differential substitution: the Leibniz
rule gives the same distribution of derivatives in the Novikov
factor and in the differential $\Var$-polynomial. It follows,
by induction on the number of compositions, that the generated
sub-operad corresponds exactly to the multilinear part of
$D\Var\<X\>$. This also gives the identification
$D\Var=\Nov\medcirc\Var$ from \cite{KolSartOrazg2019}.

Thus (i) is equivalent to the weight criterion for multilinear
polynomials on distinct variables. An arbitrary word of weight
$-1$ is obtained from such a word by identifying variables.
Conversely, if (ii) holds, projection onto the multilinear
component preserves the derived subalgebra. These observations
give the equivalence for all sets~$X$.
\end{proof}

In particular, $W_{\Var}$ is an operad such that $D\Var$ is exactly its sub-operad generated by $W_{\Var}(2)$.

\begin{corollary}[\cite{KolSartOrazg2019}]
For every binary operad $\Var$ we have
\[
D\Var = \Nov\medcirc \Var.
\]
\end{corollary}

Note that $\medcirc = \medcirc_{GK}$ here.
In some cases, we may conclude that
$\Nov\medcirc \Var = \Nov\medcirc_R \Var$. Namely, for $\Var=\Com$ this was shown in \cite{DzhLofwall2002}. For $\Var=\Lie$, the variety of Lie algebras is a 2-variety (see Section~\ref{sec:Weight}), so Corollary~\ref{cor:Dimensions} below gives $\dim D\Lie(n)=\binom{2n-2}{n-1}(n-1)!=n!\,C_{n-1}$, where $C_{n-1}$ is the Catalan number; this is the dimension of the arity~$n$ component of the operad of magmatic algebras with one operation.
Hence,
$\Nov\medcirc \Lie$ is the operad of magmatic algebras, as well as $\Nov\medcirc_R \Lie$.
For $\Var=\As$ we also have
$\Nov\medcirc \As = \Nov\medcirc_R \As$
since it was shown in \cite[Section 5]{SK_Europ}
that the relations on the operations $\prec$, $\succ $ of $\Nov\medcirc_R \As$ are enough to embed every algebra with these operations into a differential associative algebra.

\section{On the weight criterion}\label{sec:Weight}

Let $\Var $ be a binary operad generated by a finite set of operations $\mu_i = x_1\circ_i x_2$, $i\in I$. We assume that for every $i\in I$ there exists
$j\in I$ such that $(12)\mu_i = x_2\circ_i x_1$ is equal to $\pm \mu_j$.
We may and will choose these generators to form a linear basis of
$\Var(2)$.

\begin{theorem}\label{thm:WeightCriterion}
The following conditions are equivalent:
\begin{itemize}
\item [(i)] $\Nov \otimes \Var = \Nov \medcirc \Var $;
\item [(ii)] For every $i,j\in I$ there exists
$F\in (\Bbbk(13)+\Bbbk (23))(\Var(2)\vec\otimes \Var(2))$
such that the image of
\[
(\mu_i\vec\otimes \mu_j) - F
\in \mathcal F_{\Var(2)} (3)
\]
is zero in $\Var(3)$.
\end{itemize}
\end{theorem}

Practically, the condition (ii) means that
for every $i,j\in I$ there is an identity that holds on $\Var $ of the form
\begin{equation}\label{eq:ToInside}
(x_1\circ_j x_2)\circ_i x_3=\sum\limits_{k,l\in I}
\bigl(\alpha^{i,j}_{k,l} (x_3\circ_k x_2)\circ _l x_1 +
\beta^{i,j}_{k,l} (x_1\circ _k x_3)\circ _l x_2\bigr),
\end{equation}
for some
$\alpha^{i,j}_{k,l},\beta^{i,j}_{k,l}\in \Bbbk $.

For example, it works for $\Var = \Nov$ from Example~\ref{exmp:Nov}:
\[
\nu_1\vec\otimes \nu_1 = (x_1x_2)x_3 = (x_1x_3)x_2
\]
by \eqref{eq:RCom}, and
\[
\nu_2\vec\otimes \nu_1 = x_3(x_1x_2)
=
x_1(x_3x_2) + (x_3x_1)x_2 - (x_1x_3)x_2
\]
by \eqref{eq:LSym}.
It was shown in \cite{KolSartSIGMA2024} that (i) holds for
$\Var = \Nov$. Now we prove the general statement.

\begin{proof}
(i) $\Rightarrow $ (ii).
 Choose a linear basis
 $\nu_1$, $\nu_2=(12)\nu_1$
 of $\Nov(2)$, $\nu_1 = x_1x_2$, $\nu_2=x_2x_1$.
Denote $D\Var(2) = \Nov(2)\otimes \Var(2)$.

If (i) holds then, in particular,
there is an $S_3$-invariant surjective map
\[
\pi(3):
\mathcal F_{D\Var(2)}(3)
\to \Nov(3)\otimes \Var(3)\cong  D\Var (3),
\]
sending
$(\nu_p\otimes \mu_i)\vec\otimes (\nu_q\otimes \mu_j)$ to
$(\nu_p\vec\otimes \nu_q) \otimes (\mu_i\vec\otimes \mu_j)$,
$i,j\in I$, $p,q=1,2$,
such that the kernel of $\pi(3) $
consists of all
identities of degree 3 that hold on $D\Var $.

Given a pair $i,j\in I$,
suppose $\hat f \in \mathcal F_{D\Var(2)}(3)$ is a pre-image of the element
\[
f = ((13)(\nu_1\vec\otimes \nu_1))\otimes (\mu_i\vec\otimes \mu_j) \in \Nov(3)\otimes \Var(3).
\]
In general form,
\[
\hat f = \sum\limits_{p,q=1}^2 \sum\limits_{k,l\in I}
(a_{kl}^{pq} +b_{kl}^{pq}(13) + c_{kl}^{pq}(23))
 (\nu_p\otimes \mu_k)\vec\otimes (\nu_q\otimes \mu_l),
\]
where the coefficients $a,b,c$ are some scalars from $\Bbbk $.
Therefore,
\begin{equation}\label{eq:ImageF}
f
=
\pi(3)(\hat f)
=
\sum\limits_{p,q=1}^2 \sum\limits_{k,l\in I}
(a_{kl}^{pq} +b_{kl}^{pq}(13) + c_{kl}^{pq}(23))
\bigl((\nu_p\vec\otimes\nu_q)
\otimes(\mu_k\vec\otimes\mu_l)\bigr).
\end{equation}
Here the permutations act diagonally on the two tensor factors.
This is an equality in $\Nov(3)\otimes \Var(3)$.
Using the basis of $\Nov(3)$ from \cite{DzhLofwall2002} (see also Example~\ref{exmp:Nov}), we obtain a series of equations in $\Var(3)$.
Choose a basis of $\Nov(3)$ which consists of
\[
\nu_p\vec\otimes \nu_q,
\
(13)(\nu_p\vec\otimes \nu_p),\quad p,q=1,2,
\]
six elements in total. The remaining elements
of $\mathcal F_{\Nov(2)}(3)$ turn into the following combinations in $\Nov(3)$:
\begin{align*}
    (13)(\nu_1\vec\otimes \nu_2) & {}= (\nu_1\vec\otimes \nu_2),
    \\
    (13)(\nu_2\vec\otimes \nu_1) & {}=
    (\nu_1\vec\otimes \nu_1) + (\nu_2\vec\otimes \nu_1) - (13)(\nu_1\vec\otimes \nu_1),\\
    (23)(\nu_1\vec\otimes \nu_1) & {}=
    (\nu_1\vec\otimes \nu_1),\\
    (23)(\nu_1\vec\otimes \nu_2) & {}=
    (13)(\nu_1\vec\otimes \nu_1),\\
    (23)(\nu_2\vec\otimes \nu_1) & {}=
     (13)(\nu_2\vec\otimes \nu_2) - (\nu_1\vec\otimes \nu_1) + (\nu_1\vec\otimes \nu_2),
    \\
    (23)(\nu_2\vec\otimes \nu_2) & {}=
    (\nu_2\vec\otimes \nu_2)  + (\nu_1\vec\otimes \nu_2) -(13)(\nu_1\vec\otimes \nu_1) .
\end{align*}
Collect those terms in the right-hand side of \eqref{eq:ImageF} that contain the first tensor factor
$(13)(\nu_1\vec\otimes \nu_1)$.
As a result, \eqref{eq:ImageF} implies
\[
\mu_i\vec\otimes \mu_j =
\sum\limits_{k,l\in I}
\big( (b^{11}_{kl}-b_{kl}^{21}) (13) + (c_{kl}^{12}-c^{22}_{kl})(23) \big) (\mu_k\vec\otimes \mu_l)
\]
in $\Var(3)$, as desired.

(ii) $\Rightarrow$ (i).
We use Proposition~\ref{prop:WeightHadamard}. It is enough to show that the weight criterion determines the elements of $D\Var\<X\>\subset \Var\<X^{(\omega )}\>$.

Suppose (ii) holds for $\Var $.
Consider a monomial $u$
from
the free algebra $\Var\<X^{(\omega )}\>$:
this is a non-associative word in the alphabet
$X^{(\omega )}$. Let us say that the
letters $x\in X\subset X^{(\omega )}$
are ``naked''; other letters of the form
$x^{(s)}$, $s>0$, are said to be ``derived''.

\begin{lemma}\label{lem:Wedding}
Suppose a monomial $u\in \Var\<X^{(\omega )}\>$
contains both naked and derived letters.
Then $u$ may be presented as a linear combination of monomials $u_k$, $k=1,2,\dots$, such that every $u_k$ contains a subword of the form
$(x_k \circ_{i_k} y_k^{(s_k)})$,
$x_k,y_k\in X$, $i_k\in I$, $s_k>0$
(by the symmetry of the set of operations, a subword $(y_k^{(s_k)}\circ_{j} x_k)$ is of this form as well).
\end{lemma}

\begin{proof}
The degree of $u$ with respect to $X^{(\omega )}$ is greater than 1 since it contains at least one naked and one derived letter. Proceed by induction on $n = \deg u$. The case $n=2$ is obvious. If $n=3$ then, up to the symmetry of the operations,
$u = (x_1^{(s_1)}\circ _j x_2^{(s_2)})\circ _i x_3^{(s_3)}$, $s_1,s_2,s_3\ge 0$.
Suppose $x_1^{(s_1)}\circ _j x_2^{(s_2)}$
is not of the desired form (up to permutation of factors). Then either $s_3=0$, $s_1,s_2>0$, or
$s_3>0$, $s_1=s_2=0$. In both cases
the relation \eqref{eq:ToInside} allows us to rewrite the monomial $u$ in the desired form.

Now let $n\ge 4$. Then $u = u_1\circ_i u_2$,
$i\in I$, $u_1$ and $u_2$ are of smaller degree.
The worst case (when we cannot apply induction immediately) is when $u_1$ contains only naked letters and $u_2$ contains only derived letters.
Indeed, if one factor contains both kinds of letters, induction applies
to that factor. Otherwise the factors are pure and of different kinds;
by symmetry we may assume the indicated order.

At least one of $u_1$ or $u_2$ is a product of two shorter words, say, $u_1 = v_1\circ_j v_2$.
Then apply \eqref{eq:ToInside} to
\[
u = (v_1\circ_j v_2)\circ_i u_2
\]
to get a linear combination of words like
$(u_2\circ_k v_2)\circ _l v_1$,
$(v_1\circ_k u_2)\circ _l v_2$, $k,l\in I$.
The factors
$(u_2\circ_k v_2)$
and
$(v_1\circ_k u_2)$
contain both naked and derived letters,
so the induction hypothesis applies.
In the same way, the case $u_2 = v_1\circ_j v_2$
is processed.
For this, interchange the two outer factors using the symmetry of the
operations and apply the same argument. The roles of naked and derived
letters may be interchanged. In either case the mixed factor to which
induction applies has degree strictly smaller than~$n$.
\end{proof}

\begin{lemma}\label{lem:Rewriting}
Suppose a monomial $u\in \Var\<X^{(\omega )}\>$
is of weight $-1$.
Then $u$ belongs to $D\Var \<X\>$.
\end{lemma}

\begin{proof}
We have to show that $u$
is generated by $X$ in $\Var\Der\<X\>$
relative to the operations $x\prec_i y = x\circ_i d(y)$ and $x\succ_i y = d(x)\circ_i y$, $i\in I$.
(Note that if $x\circ_i y =\pm y\circ_j x$
then $x\prec_i y = \pm y\succ _j x$, so the ``duplicated'' set of operations $\prec_i$, $\succ_i$ meets our initial assumption on symmetry.)

If $\deg u=1,2$ then the statement is obvious.
For $\deg u\ge 3$, proceed by induction on the
{\em potential} of $u$. The latter (denoted
$\pt(u)$) is defined as follows \cite{SK_Europ}.
Given a non-associative word
\[
u = (x_1^{(s_1)}\circ_{i_1}\dots \circ_{i_{n-1}} x_n^{(s_n)}),\quad x_l\in X,\ s_l\ge 0,
\]
the potential of $u$ is a polynomial in $\mathbb Z_+[t]$:
\[
\pt(u) = \sum\limits_{k\ge 0} N_k t^k,
\]
where $N_k$ is the number of indices $l$ such that $s_l=k$. For example,
\[
\pt((((x_1'\circ_{i_1} x_2'')\circ_{i_2} (x_3\circ_{i_3} x_4))\circ_{i_4} x_5'))= 2+2t+t^2.
\]
Note that the potential is defined as a function on all non-associative words in $X^{(\omega )}$, but it is a well-defined function on the nonzero monomials from $\Var\<X^{(\omega )}\>$ since for a multilinear variety the ideal of identities is $\pt$-homogeneous.

Equivalently, the multihomogeneous grading by the letters $X^{(\omega)}$ induces a grading by these potentials. We may discard zero summands; every application of a multilinear identity preserves the potential of the remaining words.

We will compare potentials as follows:
$\pt(u)<\pt(v)$ if and only if
the leading coefficient of
$\pt(v)-\pt(u)$
is positive. Note that if we rewrite $u$ into a linear combination of $u_k$ as in Lemma~\ref{lem:Wedding} then $\pt(u_k)=\pt(u)$
for all $k=1,2,\dots $.

Fix a bound $n$ on the degree and prove the statement for all
alphabets and all words of degree at most~$n$ simultaneously.
A word of weight $-1$ and degree at most~$n$ has total derivative
order at most $n-1$, so there are only finitely many possible
potentials. The replacements below never increase the degree.
This justifies induction with the indicated order.

Suppose $p(t)=\pt(u)$, $\deg u\ge 3$, $\wt(u)=-1$,
and the lemma is true for all words of smaller potential. Then $u$ must contain both naked
and derived letters.
By Lemma~\ref{lem:Wedding} we may assume
that $u$ contains a subword $(x\circ_i y^{(s)})$,
$x,y\in X$,
$s\ge 1$.
Replace this subword with
\[
(x\prec _i y)^{(s-1)} -
\sum\limits_{l=1}^{s-1} \binom{s-1}{l} x^{(l)}\circ_i y^{(s-l)},
\]
then $u$ becomes a linear combination of non-associative monomials in the extended alphabet
$(X_1)^{(\omega )}$,
where
\[
X_1 = X\cup \{x\prec_i y \mid x,y\in X,\ i\in I\}.
\]
Each summand in this linear combination
is of weight $-1$ and has smaller potential
(all relative to the new alphabet).
Hence, by induction $u$ may be generated by $X_1$
by means of the derived operations $\prec_i$, $\succ_i $. Since $X_1$ itself is generated by $X$ by the same operations, the claim follows.

To make the change of alphabet precise, introduce a new independent
letter $z$ and use the alphabet
$Y=X\sqcup\{z\}$. Let $u_0$ be obtained by replacing the selected
subword by $z^{(s-1)}$, and let $u_l$ be obtained by replacing it
by $x^{(l)}\circ_i y^{(s-l)}$, $1\le l\le s-1$.
In the free differential algebra on $Y$, consider the evaluation
homomorphism which fixes $X$ and sends $z$ to $x\prec_i y$.
The displayed identity says that $u$ is the image of
\[
u_0-\sum_{l=1}^{s-1}\binom{s-1}{l}u_l.
\]
All these words have weight $-1$. Their potentials are
\[
\pt(u_0)=p(t)-1-t^s+t^{s-1},\qquad
\pt(u_l)=p(t)-1-t^s+t^l+t^{s-l}.
\]
Each is strictly smaller than $p(t)$, since the highest changed coefficient, that of $t^s$, decreases. For $s=1$ the sum is empty and $\pt(u_0)=p(t)-t$. The induction hypothesis over the alphabet $Y$ now implies that these words belong to $D\Var\<Y\>$. Evaluation preserves the derived operations, and the image of $z$ already belongs to $D\Var\<X\>$. Therefore the image of every summand, and hence $u$, belongs to $D\Var\<X\>$.
\end{proof}

To complete the proof of the theorem, recall that all elements from $D\Var\<X\>$ are of weight $-1$. If (ii) holds then by Lemma~\ref{lem:Rewriting} the converse is true. Hence, the weight criterion works for $\Var $ and thus $\Nov\otimes \Var = \Nov\medcirc \Var$.
\end{proof}

Following \cite{Anderson1971,AndersonKleinfeld1976}, we use the following terminology.

\begin{definition}
A variety $\Var$ is called a \emph{2-variety} if, for every $A\in\Var$ and every ideal $J$ of $A$, the subspace
\[
J^2=\operatorname{span}\{a\circ_i b\mid a,b\in J,\ i\in I\}
\]
is an ideal of~$A$.
\end{definition}

The following characterization is a multi-operation version of the classical criterion for 2-varieties; see \cite{Anderson1971,AndersonKleinfeld1976}.

\begin{proposition}\label{prop:TwoVariety}
Condition~{\rm(ii)} of Theorem~\ref{thm:WeightCriterion} is
equivalent to each of the following conditions:
\begin{itemize}
\item[(a)] $\Var$ is a 2-variety;
\item[(b)] the variety governed by $\Var_q$ is a 2-variety.
\end{itemize}
\end{proposition}

\begin{proof}
The equivalence of condition~{\rm(ii)} and~(a) follows by the
standard argument for 2-varieties
\cite{Anderson1971,AndersonKleinfeld1976}, applied to the family
of operations $\circ_i$ and their opposites. Namely, if (ii) holds,
$J$ is an ideal of $A\in\Var$, $a,b\in J$ and $c\in A$, then
\eqref{eq:ToInside} expresses $(a\circ_j b)\circ_i c$ through products
$(c\circ_k b)\circ_l a$ and $(a\circ_k c)\circ_l b$, which lie in $J^2$;
by the symmetry of the set of operations, $J^2$ is an ideal.
Conversely, let $A$ be the free $\Var$-algebra on $x_1,x_2,x_3$ and $J$
the ideal generated by $x_1,x_2$. The multilinear component of $J^2$
is spanned by the products in which $x_1$ or $x_2$ is the outer factor,
so $(x_1\circ_j x_2)\circ_i x_3\in J^2$ gives exactly~\eqref{eq:ToInside}. Finally, $\Var$ and
$\Var_q$ have the same identities of degree~$3$, which proves the
equivalence with~(b).
\end{proof}

\begin{corollary}\label{cor:Dimensions}
If $\Var_q$ is a 2-variety, then
\[
\dim D\Var(n)=\binom{2n-2}{n-1}\dim\Var(n),\qquad n\ge1.
\]
\end{corollary}

\begin{proof}
Theorem~\ref{thm:WeightCriterion} and
Proposition~\ref{prop:TwoVariety} give
$D\Var(n)\cong\Nov(n)\otimes\Var(n)$.
The basis in Proposition~\ref{prop:WeightHadamard} is indexed by
the weak compositions of $n-1$ into $n$ parts, whose number is
$\binom{2n-2}{n-1}$.
\end{proof}

\begin{corollary}\label{cor:Speciality}
If $\Var_q$ is a 2-variety, then every $D\Var$-algebra is special.
\end{corollary}

\begin{proof}
The weight criterion follows from Theorem~\ref{thm:WeightCriterion}
and Proposition~\ref{prop:TwoVariety}. Apply
\cite[Theorem~3.5]{KolSartSIGMA2024}.
\end{proof}

\begin{example}
The criterion holds for $\Com$, $\As$, and $\Lie$.
For the associative cases use associativity and the opposite
operations. For $\Lie$, the required identity is
\[
[[x_1,x_2],x_3]=[[x_3,x_2],x_1]+[[x_1,x_3],x_2].
\]
For $\Nov$, the computations preceding the proof of
Theorem~\ref{thm:WeightCriterion}, together with the permutation
$(12)$ of the inner arguments, cover all pairs of generators.
The criterion also holds for $\Acaa$, since
$(x_1x_2)x_3=(x_2x_3)x_1$.
\end{example}

\begin{example}\label{exmp:NoCubicRelations}
Suppose $\Var$ has one commutative binary operation and no
identities of degree~3 beyond commutativity. Then
$\Var_q=\ComMag$, the free operad on one commutative binary
operation. We have $\dim\Var(3)=3$ and $\dim\Nov(3)=6$, whereas
$\dim(\Nov(2)\otimes\Var(2))=2$. Consequently,
\[
\dim(\Nov\medcirc\Var)(3)\le 12<18
=\dim(\Nov(3)\otimes\Var(3)).
\]
The weight criterion therefore fails. This applies in particular to the variety of Jordan algebras, whose additional defining identity has degree~4. The failure of the weight criterion in this case does not contradict speciality of all derived Jordan algebras; see \cite[Proposition~3.7]{KolSartSIGMA2024}.
\end{example}

The criterion of Theorem~\ref{thm:WeightCriterion} also admits a quantitative form. Let $\Var$ be an arbitrary binary
operad, and
\[
V=\Var(2),\qquad M=\mathcal F_V(3),\qquad
\pi:M\longrightarrow\Var(3),\qquad R=\ker\pi.
\]
Write
\[
M=U_1\oplus U_2\oplus U_3,\qquad
U_3=V\vecotimes V,\quad U_1=(13)U_3,\quad U_2=(23)U_3.
\]
Thus $U_s$ corresponds to the products in which $x_s$ is the outer
factor. Denote by $p_s:M\to U_s$ the corresponding projection,
and put $A_s=\pi(U_s)$. Hence,
\[
A_3=\operatorname{span}\{(x_1\circ_jx_2)\circ_i x_3\mid i,j\in I\}\subseteq \Var(3).
\]
We call
\[
\delta(\Var)=\dim\bigl(U_3/p_3(R)\bigr)
\]
the \emph{cubic defect} of $\Var$. Since $R$ is $S_3$-invariant,
the three spaces $U_s/p_s(R)$ have the same dimension. By
Theorem~\ref{thm:WeightCriterion}, the weight criterion holds if
and only if $\delta(\Var)=0$.

Let $N=\Nov(3)$. In the differential realization of the free
Novikov algebra, choose the basis
\[
e_s=x_s''x_jx_k,\qquad f_s=x_sx_j'x_k',
\qquad \{s,j,k\}=\{1,2,3\}.
\]
For $s=1,2,3$, set
\begin{equation}\label{eq:CubicNovTree}
N_s=\operatorname{span}\{f_j,f_k,f_s+e_j,f_s+e_k\},
\qquad \{s,j,k\}=\{1,2,3\}.
\end{equation}

\begin{lemma}\label{lem:CubicDerivedImage}
Under the identification of
Proposition~\ref{prop:WeightHadamard},
\begin{equation}\label{eq:CubicDerivedImage}
D\Var(3)=\sum_{s=1}^3N_s\otimes A_s
\subseteq N\otimes\Var(3).
\end{equation}
More precisely, $N_s\otimes A_s$ is the image of the tree summand
of $\mathcal F_{D\Var(2)}(3)$ with outer argument $x_s$.
Moreover, $N_s+N_t=N$ for $s\ne t$, and $f_s\notin N_s$.
\end{lemma}

\begin{proof}
For a fixed inner pair $x_j,x_k$, the four compositions of binary
Novikov operations are
\[
(x_j\circ x_k)\circ x_s=f_j,\qquad
(x_k\circ x_j)\circ x_s=f_k,
\]
\[
x_s\circ(x_j\circ x_k)=f_s+e_k,\qquad
x_s\circ(x_k\circ x_j)=f_s+e_j.
\]
They form a basis of $N_s$. The labels from $V$ at the two vertices
are independent of the Novikov labels, so the image of the
corresponding tree summand is exactly $N_s\otimes A_s$.
Summing over the three inner pairs gives
\eqref{eq:CubicDerivedImage}.
The remaining assertions follow directly from
\eqref{eq:CubicNovTree}: the sum of any two of these spaces
contains all $f_i$, and then all $e_i$, whereas $f_s$ itself
cannot be a linear combination of the four displayed generators
of $N_s$.
\end{proof}

\begin{theorem}\label{thm:CubicDefect}
For every binary operad $\Var$,
\begin{equation}\label{eq:CubicDefectDimension}
\dim\bigl((\Nov(3)\otimes\Var(3))/D\Var(3)\bigr)
=6\delta(\Var).
\end{equation}
In particular,
\[
\dim D\Var(3)=6\bigl(\dim\Var(3)-\delta(\Var)\bigr).
\]
\end{theorem}

\begin{proof}
Put $\overline N_s=N/N_s$. By
Lemma~\ref{lem:CubicDerivedImage}, the quotient on the left of
\eqref{eq:CubicDefectDimension} is naturally isomorphic to
\[
\frac{N\otimes M}
{N\otimes R+\sum_{s=1}^3N_s\otimes U_s}.
\]
First factor out the second summand in the denominator. This gives
\[
H=\bigoplus_{s=1}^3\overline N_s\otimes U_s.
\]
Write $\overline e_s,\overline f_s$ for the images of $e_s,f_s$
in $\overline N_s$. They form a basis of this two-dimensional
space. For $t\ne s$, the image of $f_s$ in $\overline N_t$ is
zero, and the image of $e_s$ is $-\overline f_t$.
Consequently, for every $r\in R$, the images in $H$ of
$f_s\otimes r$ and $e_s\otimes r$ are, respectively,
\[
\overline f_s\otimes p_s(r),\qquad
\overline e_s\otimes p_s(r)
-\sum_{t\ne s}\overline f_t\otimes p_t(r).
\]
The first family contains every
$\overline f_s\otimes p_s(r)$. Subtracting such elements from the
second family also gives every $\overline e_s\otimes p_s(r)$.
Thus the image of $N\otimes R$ in $H$ is precisely
\[
\bigoplus_{s=1}^3\overline N_s\otimes p_s(R).
\]
It follows that
\begin{equation}\label{eq:CubicDefectDecomposition}
\frac{\Nov(3)\otimes\Var(3)}{D\Var(3)}
\cong
\bigoplus_{s=1}^3\overline N_s\otimes\bigl(U_s/p_s(R)\bigr).
\end{equation}
Each summand has dimension $2\delta(\Var)$, which proves the
statement.
\end{proof}

\begin{corollary}\label{cor:CubicWitness}
If the weight criterion fails for $\Var$, then there exist
$i,j\in I$ such that
\[
(x_1'\circ_jx_2')\circ_i x_3
\notin D\Var\<x_1,x_2,x_3\>.
\]
This polynomial has weight $-1$.
\end{corollary}

\begin{proof}
Since $p_3(R)\ne U_3$, some basis element
$u=\mu_i\vecotimes\mu_j$ does not belong to $p_3(R)$.
Under Proposition~\ref{prop:WeightHadamard}, the displayed
polynomial corresponds to $f_3\otimes\pi(u)$. Its image under
\eqref{eq:CubicDefectDecomposition} is
$\overline f_3\otimes(u+p_3(R))$, which is nonzero. Hence the displayed polynomial does not belong to
$D\Var\<x_1,x_2,x_3\>$.
\end{proof}

\section{On the Dong property}\label{sec:Dong}

Suppose $A$ is a linear space equipped
with a family of bilinear operations
$\circ _i$, $i\in I$. As above, we assume
the set $I$ is finite and the set of
operations is symmetric, i.e.,
for every $i\in I$ there is $j\in I$ such that
$y\circ _i x = \pm x\circ_j y$
for all $x,y\in A$.

A formal distribution over $A$ is a formal power series of the form
\[
a(z) = \sum\limits_{n\in \mathbb Z} a(n)z^{-n-1},
\quad a(n)\in A.
\]
The space of all formal distributions over $A$
is denoted $A[[z,z^{-1}]]$.

Two formal distributions
$a(z),b(z)\in A[[z,z^{-1}]]$ form a local pair
if there exists $N\in \mathbb Z_+$
such that for all $i\in I$
\begin{equation}\label{eq:Locality}
(a(w)\circ_i b(z)) (w-z)^N =0
\end{equation}
as a series in the space
$A[[z,z^{-1}, w,w^{-1}]]$.

In terms of coefficients, \eqref{eq:Locality}
means that
\begin{equation}\label{eq:LocalityCoeff}
\sum\limits_{s\ge 0} (-1)^s\binom{N}{s} a(n-s)\circ_i b(m+s) =0
\end{equation}
for all $i\in I$ and for all $n,m\in \mathbb Z$.

The condition \eqref{eq:LocalityCoeff}
may be considered as a recurrence relation
for the sequence $a(p)\circ_i b(q)$,
$p+q=\mathrm{const}$.
Using this relation,
for every $p,q \in \mathbb Z$ and for every
$i\in I$
the product
$a(p)\circ_i b(q)\in A$ may be expressed as a linear
combination of
\[
a(0)\circ_i  b(p+q),\ a(1)\circ_i b(p+q-1),\
\dots, \
a(N-1)\circ_i b(p+q-N+1).
\]
The standard delta-function expansion (see \cite{Kac1}) is
\[
a(w)\circ_i b(z) = \sum\limits_{n=0}^{N-1}
 c_{i(n)}(z) \dfrac{1}{n!} \partial_z^n
 \delta(w-z),
\]
where $\delta(w-z) = \sum\limits_{k+l=-1}w^kz^l$
is the formal delta-function, $c_{i(n)}(z) \in A[[z,z^{-1}]]$
is given by the formula
\[
c_{i(n)}(z) = \Res\limits_{w=0} (a(w)\circ_i b(z))(w-z)^n,
\]
that is, in terms of coefficients,
\[
c_{i(n)}(m) = \sum\limits_{s=0}^n (-1)^s\binom{n}{s} a(n-s)\circ_i b(m+s) \in A,
\]
for every $m\in \mathbb Z$.
The series $c_{i(n)}(z)$ are called
conformal products of $a(z)$ and $b(z)$:
\[
c_{i(n)}(z) = (a\circ_{i(n)} b)(z),\quad i\in I,\ n\in \mathbb Z_+.
\]

We say that an algebra $A$ satisfies the Dong property if for every triple of pairwise local formal distributions
$a(z),b(z),c(z)\in A[[z,z^{-1}]]$
the series $(a\circ_{i(n)} b)(z)$ and $c(z)$
form a local pair for all $i\in I$, $n\in \mathbb Z_+$.
If $A$ satisfies the Dong property, then every subset of pairwise local distributions {\em generates} a conformal algebra of formal distributions in $A[[z,z^{-1}]]$; see \cite{Kac1}.

If every algebra in a variety $\Var $ satisfies the Dong property, then we say that $\Var $ is a Dong variety, and the operad governing $\Var$ is said to be a Dong operad.
For example, $\Var=\Lie $ (Lie algebras)
and $\Var =\As$ (associative algebras) are Dong varieties.

The operad $\Var$ need not be quadratic; for example, the Jordan operad has defining relations of degree~4. Nevertheless, whether the Dong property holds is determined by the identities of degrees~2 and~3: $\Var$ is a Dong operad if and only if its quadratic part
$\Var_{q}$ (defined by the quadratic relations of $\Var$) is a Dong operad.
We justify this passage to the quadratic part after the following criterion.

\begin{theorem}[\cite{KolSart_DongOperads}]\label{thm:DongCriterion}
A binary quadratic operad $\Var$
generated by operations $\mu_i$, $i\in I$,
is a Dong operad if and only if
the elements $\mu_i^\vee \vec\otimes \mu_j^\vee$,
$i,j\in I$, are linearly independent in the
Koszul dual operad $\Var^!$.
\end{theorem}

As in Section~\ref{sec:Weight}, the chosen binary generators form a
linear basis. Here $\mu_i^\vee$ denotes the dual basis with the sign-twisted
$S_2$-action fixed in Section~\ref{sec:Preliminaries}.
The statement is the combination of the sufficient and necessary
conditions proved in \cite{KolSart_DongOperads}.

The criterion extends to an arbitrary binary operad via its quadratic part. If $\Var_q$ is Dong, then so is $\Var$,
since every $\Var$-algebra is a $\Var_q$-algebra. For the converse,
the necessity argument in \cite{KolSart_DongOperads} uses the free
algebra on the coefficient symbols of three distributions, modulo
homogeneous locality relations of degree~2. Failure of locality
of a conformal product is detected by elements of degree~3.
One may also factor out all components of degrees at least~4:
the same degree~3 elements remain nonzero. The resulting algebra
satisfies every identity of degree at least~4 and is therefore a
$\Var$-algebra whenever its identities of degrees~2 and~3 are
those of $\Var_q$. This proves that $\Var$ is Dong if and only
if $\Var_q$ is Dong.

We now relate the Dong property of $\Var$ to the weight criterion for $D\Var$.

\begin{theorem}\label{thm:DongWeight}
Let $\Var $ be a binary quadratic operad.
Then the following statements are equivalent:
\begin{itemize}
    \item [(i)] $\Var $ is a Dong operad;
    \item[(ii)] $\Nov\otimes \Var = \Nov\medcirc \Var $.
\end{itemize}
\end{theorem}

\begin{proof}
Write $\Var=\mathcal P(V,R)$ and set
\[
U=V\vecotimes V,\qquad
W=(13)U\oplus(23)U,
\qquad \mathcal F_V(3)=U\oplus W.
\]
Let $p_U:\mathcal F_V(3)\to U$ be the projection along $W$.
Condition~(ii) of Theorem~\ref{thm:WeightCriterion} says precisely
that
\begin{equation}\label{eq:ProjectionCriterion}
p_U(R)=U.
\end{equation}
Indeed, for $u\in U$ this condition means that $u-w\in R$ for
some $w\in W$, which is the required rewriting relation.

The pairing defining the Koszul dual is nondegenerate on each
tree summand, and distinct summands are orthogonal. Hence
$U^\vee=V^\vee\vecotimes V^\vee$ pairs only with $U$, and
\[
R^\perp\cap U^\vee
=\{\lambda\in U^\vee\mid
\langle\lambda,p_U(r)\rangle=0\text{ for every }r\in R\}.
\]
Since $U$ is finite-dimensional, \eqref{eq:ProjectionCriterion}
is equivalent to $R^\perp\cap U^\vee=0$. The latter means that
the canonical map
\[
U^\vee\longrightarrow
\mathcal F_{V^\vee}(3)/R^\perp=\Var^!(3)
\]
is injective. By Theorem~\ref{thm:DongCriterion}, this is equivalent
to~(i). Theorem~\ref{thm:WeightCriterion} identifies
\eqref{eq:ProjectionCriterion} with~(ii), completing both implications.
\end{proof}

\begin{corollary}\label{cor:AllCriteria}
For a binary operad $\Var$, the following conditions are equivalent:
\begin{itemize}
\item[(i)] the weight criterion holds for $D\Var\<X\>$ for every $X$;
\item[(ii)] $\Nov\medcirc\Var=\Nov\otimes\Var$;
\item[(iii)] $\Var_q$ governs a 2-variety;
\item[(iv)] $\Var_q$ is a Dong operad;
\item[(v)] $\Var$ is a Dong operad.
\end{itemize}
\end{corollary}

\begin{proof}
Use Proposition~\ref{prop:WeightHadamard},
Theorem~\ref{thm:WeightCriterion},
Proposition~\ref{prop:TwoVariety}, and
Theorem~\ref{thm:DongWeight}, together with the fact that the
condition in degree~3 is the same for $\Var$ and $\Var_q$.
The equivalence of (iv) and (v) was proved above.
\end{proof}

\subsection{Dong property of derived operads}\label{sec:DerivedDong}

We now characterize the binary operads $\Var$ for which $D\Var$ has the Dong property.
The notation $U_s,A_s,N_s$ from Section~\ref{sec:Weight}
allows us to describe exactly when the derived operad is Dong.

For a rooted binary tree $T$ with leaves labelled by
$1,\dots,n$, let
\[
\gamma_T=\gamma_T^{\Var }:V^{\otimes(n-1)}\longrightarrow\Var(n)
\]
be the map given by composition along $T$. An ordering of the
inputs at each vertex may be chosen arbitrarily; changing it is
absorbed by the $S_2$-action on $V$.

\begin{theorem}\label{thm:DerivedDong}
For a binary operad $\Var$, the following conditions are equivalent:
\begin{itemize}
\item[(i)] $D\Var$ is a Dong operad;
\item[(ii)] the map $U_3\to\Var(3)$ is surjective;
\item[(iii)] the map $\gamma_T$ is surjective for every $n\ge2$
and every rooted binary tree $T$ with $n$ labelled leaves;
\item[(iv)] $\Var\medcirc_{GK}\mathcal Q=\Var\otimes\mathcal Q$
for every binary operad $\mathcal Q$.
\end{itemize}
\end{theorem}

Condition~(iii) is the treewise surjectivity condition appearing
in \cite[Proposition~15]{Vallette2008}. In condition~(iv), the
white product is explicitly the generated sub-operad. No assertion
about its quadraticity is included.

\begin{proof}
(i) $\Rightarrow$ (ii).
By Corollary~\ref{cor:AllCriteria} and Theorem~\ref{thm:WeightCriterion}, the image of any tree summand in $D\Var(3)$ is contained in the sum of the images of the other two summands. Lemma~\ref{lem:CubicDerivedImage} therefore gives
\[
N_s\otimes A_s
\subseteq N_t\otimes A_t+N_u\otimes A_u,
\qquad \{s,t,u\}=\{1,2,3\}.
\]
Since $f_t=x_tx_u'x_s'\in N_s$ does not belong to $N_t$, there exists a linear functional $\varepsilon\in N^*$ such that
$\langle \varepsilon , N_t\rangle =0$,
$\langle \varepsilon , f_t\rangle =1$.
For every $a\in A_s$, apply $\langle \varepsilon, \cdot\rangle \otimes\mathrm{id}_{\Var(3)}$ to
$f_t\otimes a \in N_s\otimes A_s $: as a result, we obtain $a\in A_u$. Hence $A_s\subseteq A_u$. The symmetry implies $A_1=A_2=A_3$, and recall that their sum is $\Var(3)$. Therefore,
$A_s=\Var(3)$ for every $s=1,2,3$, and (ii) follows.

(ii) $\Rightarrow$ (i).
The $S_3$-action gives $A_1=A_2=A_3=\Var(3)$. Since $N_s+N_t=N$ for distinct $s,t$, the images of any two tree summands already span $N\otimes\Var(3)$. In particular, each tree summand in $D\Var(3)$ can be rewritten through the other two. Apply Theorem~\ref{thm:WeightCriterion} to $D\Var$ and then Corollary~\ref{cor:AllCriteria}.

(ii) $\Rightarrow$ (iii).
Condition~(ii), together with its permutations,
says that (iii) holds for $n=3$, that is, $\Var$ is weakly associative in the sense of \cite{Strohmayer}. Hence, (iii) holds for every $n\ge 2$.

(iii) $\Rightarrow$ (iv).
This was proved in \cite{Vallette2008}. Indeed, the arity~$n$ component of $\Var\medcirc_{GK}\mathcal Q$ is the sum over all trees $T$ of the images $\operatorname{im}\gamma_T^{\Var}\otimes\operatorname{im}\gamma_T^{\mathcal Q}=\Var(n)\otimes\operatorname{im}\gamma_T^{\mathcal Q}$, and the images $\operatorname{im}\gamma_T^{\mathcal Q}$ span $\mathcal Q(n)$.

(iv) $\Rightarrow$ (ii).
Take $\mathcal Q=\ComMag$. Its component of arity~3 has a basis $q_1,q_2,q_3$, where $q_s$ is the monomial with outer variable $x_s$. The image of the binary-generated sub-operad in arity~3 is
\[
\bigoplus_{s=1}^3 A_s\otimes\Bbbk q_s.
\]
Equality with $\Var(3)\otimes\ComMag(3)$ implies
$A_s=\Var(3)$ for every $s=1,2,3$.
\end{proof}

\begin{corollary}\label{cor:DerivedIterations}
For every binary operad $\Var$,
\[
D^2\Var\text{ is a Dong operad}
\quad\Longleftrightarrow\quad \Var(3)=0.
\]
If $\Var(3)\ne0$, then $D^k\Var$ is not a Dong operad for any
$k\ge2$. If $\Var(3)=0$, then all $D^k\Var$, $k\ge0$, are Dong
operads. The same statements hold for the weight criterion.
\end{corollary}

\begin{proof}
Suppose $\Var(3)\ne0$. Then each of the spaces $A_s$ is nonzero due to $S_3$-symmetry.
Choose $0\ne a\in A_1$. By Lemma~\ref{lem:CubicDerivedImage},
\[
f_3\otimes a\in N_1\otimes A_1\subseteq D\Var(3).
\]
The image of the tree summand with outer argument $x_3$ in
$D\Var(3)$ is $N_3\otimes A_3$. It does not contain
$f_3\otimes a$, since $f_3\notin N_3$.
Thus condition~(ii) of Theorem~\ref{thm:DerivedDong} fails for
the operad $D\Var$. The theorem implies that $D^2\Var$ is not Dong.

There is a canonical epimorphism $D\Var\to\Var$. Indeed, the morphism $\Nov\to\Com$ sending the Novikov product to the commutative associative product induces a morphism
\[
\Nov\otimes\Var\longrightarrow\Com\otimes\Var\cong\Var.
\]
Its restriction to $D\Var$ is surjective on binary generators, and hence is an epimorphism. The Dong property passes to quotient operads. Applying this observation repeatedly shows that no $D^k\Var$ with $k\ge2$ can be Dong in the case considered. It also shows that, once the Dong property fails at any step, it cannot hold at a later step.

If $\Var(3)=0$, then $D^k\Var(3)=0$ for all $k\ge0$, since each derived component is a subspace of the corresponding Hadamard product. The condition of Theorem~\ref{thm:WeightCriterion} is then automatic. Finally, use Corollary~\ref{cor:AllCriteria} to pass between the Dong property and the weight criterion.
\end{proof}

\begin{example}
The bound in Corollary~\ref{cor:DerivedIterations} is attained for $\Var=\Com$. Both $\Com$ and $D\Com=\Nov$ are Dong operads, whereas $D^2\Com=D\Nov$ is not. Thus the weight criterion holds for $\Com$ and $D\Com=\Nov$, but fails for $D^2\Com=D\Nov$.
\end{example}

The following example shows that the two white products may differ even when one factor is $\Nov$. In particular, the weight criterion does not imply that the derived operad of a binary quadratic operad is quadratic.

\begin{proposition}\label{thm:DerivedAcaaNonquadratic}
For the operad $\Acaa$,
\[
D\Acaa = \Nov\medcirc_{GK}\Acaa=\Nov\otimes\Acaa
\]
is a Dong operad, but the canonical epimorphism
\[
\Nov\medcirc_R\Acaa\longrightarrow\Nov\medcirc_{GK}\Acaa
\]
is not injective in arity~4. Hence, $D\Acaa$ is not quadratic.
\end{proposition}

\begin{proof}
Since $\Acaa$-algebras form a 2-variety, we have the equality $\Nov\medcirc_{GK}\Acaa=\Nov\otimes\Acaa$.
Note that $\Acaa(3)$ is one-dimensional and spanned by $(x_1x_2)x_3$, so $\Acaa$ also satisfies condition~(ii) of Theorem~\ref{thm:DerivedDong}; hence $D\Acaa$ is a Dong operad.

Moreover, it is easy to check that $\Acaa(4)=0$.
Therefore $D\Acaa(4)\subseteq\Nov(4)\otimes\Acaa(4)=0$, and it suffices to show that $\mathcal W(4)\ne0$ for the quadratic operad $\mathcal W=\Nov\medcirc_R\Acaa$.

As shown by Example~\ref{exmp:White_R_Nov_Acaa}, the quadratic operad $\Nov\medcirc_R\Acaa$ is generated by one binary operation $\mu_1: (x,y)\mapsto x*y$ without prescribed symmetry, and
 $\mu_2: (x,y)\mapsto y*x$ is its opposite.
The identities derived in  Example~\ref{exmp:White_R_Nov_Acaa}
for $\mathcal W = \Nov\medcirc_R \Acaa$
are
\begin{equation}\label{eq:AcaaDerivedRight}
(x*y)*z+(x*z)*y=0,
\end{equation}
\begin{equation}\label{eq:AcaaDerivedLeft}
(x*y)*z+x*(y*z)+(y*x)*z+y*(x*z)=0.
\end{equation}

To show that $\mathcal W(4)\ne0$, we construct an explicit example of a $\mathcal W$-algebra. Let $E=\Bbbk a\oplus\Bbbk b$ and let $\omega:E\times E\to\Bbbk$ be the alternating bilinear form with $\omega(a,b)=1$. Let $\overline E$ be a copy of $E$, with $x\mapsto\overline x$ the fixed linear isomorphism. Consider the nine-dimensional graded space
\[
B=E\oplus(E\otimes E)\oplus\overline E\oplus\Bbbk c,
\]
whose summands have degrees $1,2,3,4$, respectively. Define its multiplication, for $x,y,z,t\in E$, by
\begin{equation}\label{eq:AcaaWitnessProducts}
\begin{aligned}
x*y&=x\otimes y,\\
x*(y\otimes z)&=\omega(x,y)\overline z,\\
x*\overline y=\overline y*x&=-\omega(x,y)c,\\
(x\otimes y)*(z\otimes t)&=\omega(x,z)\omega(y,t)c,
\end{aligned}
\end{equation}
and set all other products of homogeneous summands equal to zero.

It is straightforward to check that the identities \eqref{eq:AcaaDerivedRight} and \eqref{eq:AcaaDerivedLeft} hold on $B$, so $B$ is a $\mathcal W$-algebra.

Finally, \eqref{eq:AcaaWitnessProducts} gives
\[
(a*a)*(b*b)=(a\otimes a)*(b\otimes b)
=\omega(a,b)^2c=c\ne0.
\]
Consequently, the multilinear operation $(x_1*x_2)*(x_3*x_4)$ is nonzero in $\mathcal W(4)$, although its image in $D\Acaa(4)=0$ is zero.
\end{proof}

\textbf{Acknowledgments.}
This research was funded by the Committee of Science of the Ministry of Science and Higher Education of the Republic of Kazakhstan (Grant No.~AP26199089) and the Program of Fundamental Research RAS (project
FWNP-2026-0017).

We are grateful to Vladimir Dotsenko and Vsevolod Gubarev for suggesting that we distinguish between the quadratic white Manin product and the white Manin product in the sense of Ginzburg and Kapranov.


\begin{thebibliography}{99}

\bibitem{Anderson1971}
T.~Anderson,
The Levitzki radical in varieties of algebras,
Math. Ann. \textbf{194} (1971), 27--34.

\bibitem{AndersonKleinfeld1976}
T.~Anderson, E.~Kleinfeld,
A classification of 2-varieties,
Canad. J. Math. \textbf{28} (2) (1976), 348--364.

\bibitem{BalNov1985}
A.~A.~Balinskii, S.~P.~Novikov,
Poisson brackets of hydrodynamic type, Frobenius algebras and Lie algebras,
Sov. Math. Dokl. \textbf{32} (1) (1985), 228--231.

\bibitem{BCZ2017}
L.~A.~Bokut, Y.~Chen, Z.~Zhang,
Gr\"obner--Shirshov bases method for Gelfand--Dorfman--Novikov algebras,
J. Algebra Appl. \textbf{16} (1) (2017), 1750001, 22~pp.
\href{https://doi.org/10.1142/S0219498817500013}{doi:10.1142/S0219498817500013}.

\bibitem{DzhLofwall2002}
A.~S.~Dzhumadil'daev, C.~L\"ofwall,
Trees, free right-symmetric algebras, free Novikov algebras and identities,
Homology Homotopy Appl. \textbf{4} (2) (2002), 165--190.

\bibitem{GD1979}
I.~M.~Gel'fand, I.~Ya.~Dorfman,
Hamiltonian operators and algebraic structures related to them,
Funct. Anal. Appl. \textbf{13} (4) (1979), 248--262.

\bibitem{GK1994}
V.~Ginzburg, M.~Kapranov,
Koszul duality for operads,
Duke Math. J. \textbf{76} (1) (1994), 203--272.
\href{https://arxiv.org/abs/0709.1228}{arXiv:0709.1228}.

\bibitem{Kac1}
V.~G.~Kac,
Vertex algebras for beginners, second edition,
University Lecture Series, vol.~10,
American Mathematical Society, Providence, RI, 1998.

\bibitem{KolSartSIGMA2024}
P.~Kolesnikov, F.~Mashurov, B.~Sartayev,
On pre-Novikov algebras and derived Zinbiel variety,
SIGMA \textbf{20} (2024), 017, 15~pp.
\href{https://doi.org/10.3842/SIGMA.2024.017}{doi:10.3842/SIGMA.2024.017}.

\bibitem{KolSart_DongOperads}
P.~S.~Kolesnikov, B.~K.~Sartayev,
On the Dong property for a binary quadratic operad,
J. Algebra \textbf{691} (2026), 428--452.
\href{https://doi.org/10.1016/j.jalgebra.2025.11.025}{doi:10.1016/j.jalgebra.2025.11.025}.
\href{https://arxiv.org/abs/2412.20021}{arXiv:2412.20021}.

\bibitem{KolSartOrazg2019}
P.~S.~Kolesnikov, B.~Sartayev, A.~Orazgaliev,
Gelfand--Dorfman algebras, derived identities, and the Manin product of operads,
J. Algebra \textbf{539} (2019), 260--284.
\href{https://arxiv.org/abs/1903.02238}{arXiv:1903.02238}.

\bibitem{LodayVallette2012}
J.-L.~Loday, B.~Vallette,
Algebraic operads,
Grundlehren der mathematischen Wissenschaften, vol.~346,
Springer, Heidelberg, 2012.
\href{https://doi.org/10.1007/978-3-642-30362-3}{doi:10.1007/978-3-642-30362-3}.

\bibitem{Osborn92}
J.~M.~Osborn,
Novikov algebras,
Nova J. Algebra Geom. \textbf{1} (1992), 1--14.

\bibitem{ZP2024}
A.~P.~Pozhidaev, V.~N.~Zhelyabin,
On simple and semisimple finite-dimensional Novikov algebras and their automorphisms,
J. Algebra \textbf{689} (2026), 1--26.

\bibitem{SK_Europ}
B.~Sartayev, P.~Kolesnikov,
Noncommutative Novikov algebras,
Eur. J. Math. \textbf{9} (2) (2023), 35, 18~pp.
\href{https://doi.org/10.1007/s40879-023-00632-1}{doi:10.1007/s40879-023-00632-1}.


\bibitem{Strohmayer}
H. Strohmayer, Operads of compatible structures and weighted partitions, J. Pure Appl. Algebra \textbf{212}, No. 11, 2522--2534 (2008).


\bibitem{Vallette2008}
B.~Vallette,
Manin products, Koszul duality, Loday algebras and Deligne conjecture,
J. Reine Angew. Math. \textbf{620} (2008), 105--164.
\href{https://arxiv.org/abs/math/0609002}{arXiv:math/0609002}.

\end{thebibliography}
\end{document}